%% file: main.tex
\documentclass[a4paper, 11pt, cleardoublepage=empty,bibliography=totoc]{scrartcl} 
\input{commands.tex}
\usepackage{amssymb}
\usepackage[all]{xy}

\begin{document}
\pagestyle{plain}

\begin{center}
\Large{\textbf{Calabi-Yau Deformation Quantization }}
\end{center}
\hspace{0.5cm}
\begin{center}
    Jakob Ulmer
  \end{center}
\hspace{1cm}

\textbf{Abstract:} We record the Calabi-Yau version of Kontsevich's formality morphism from deformation quantization. As a special case we find that any unimodular holomorphic Poisson Calabi-Yau has a canonical closed deformation quantization of its resolved algebra of functions. A broader motivation is to find an explanation of Kontsevich's deformation quantization in the setting of Calabi-Yau categories via Sen-Zwiebach's string vertices.

\section{Introduction}

\input{intro}

\input{cyform}

\bibliography{refs}
\bibliographystyle{alpha}
\end{document}

%% file: commands.tex
\usepackage[left=2.5cm, right=2.5cm, top=0.5cm, bottom=2.5cm]{geometry}
\usepackage[colorlinks,
pdfpagelabels,
pdfstartview = FitH,
bookmarksopen = true,
bookmarksnumbered = true,
linkcolor = black,
plainpages = false,
hypertexnames = false,
citecolor = black] {hyperref}
\usepackage{adjustbox}
\usepackage{amsmath,amssymb,mathrsfs,amsfonts}

\usepackage{bm}
\newcommand*{\B}[1]{\ifmmode\bm{#1}\else\textbf{#1}\fi}
\usepackage[english]{babel}
\usepackage{tikz-cd}
\usepackage{amstext}
\usepackage{pdfpages}
\usepackage{mathtools}
\usepackage{stmaryrd}
\usepackage{xcolor}
\usepackage{xfrac} 
\usepackage[utf8]{inputenc}
\usepackage{url} 
\usepackage{pgfplots}
\usepackage[autooneside=false,headsepline,markcase=noupper]{scrlayer-scrpage}
\usepackage{blindtext}
\usepackage{enumitem}
\usepackage{verbatim}
\usepackage{graphicx}
\usepackage{mathtools}
\usepackage{makeidx}
\usepackage{tikz}
\usetikzlibrary{babel}

\addto\captionsenglish{}

\usetikzlibrary{%
	matrix,%
	calc,%
	arrows%
}

\definecolor{mycolor}{RGB}{196,19,47}
\definecolor{mygray}{gray}{0.4}

\setkomafont{section}{\normalfont\LARGE\bfseries}
\setkomafont{subsection}{\normalfont\large\bfseries}
\setkomafont{pagehead}{\normalfont\bfseries}
\setkomafont{title}{\normalfont\bfseries\Large}

\usepackage{chngcntr}
\counterwithin{equation}{section}
\usepackage{amsthm}

\theoremstyle{remark}
\newtheorem{remark}[equation]{Remark}

\theoremstyle{definition}
\newtheorem{df}[equation]{Definition}
\newtheorem{cons}[equation]{Construction}

\theoremstyle{plain}

\newtheorem*{corollary_nn_b}{Corollary B}
\newtheorem*{corollary_nn_c}{Corollary C}
\newtheorem*{thm_nn}{Theorem}
\newtheorem{theorem}[equation]{Theorem}

\newtheorem{corollary}[equation]{Corollary}

\newtheorem{lemma}[equation]{Lemma}

\newtheorem*{thm_nn_a}{Theorem A}
\newtheorem{prop}[equation]{Proposition}

\clearpairofpagestyles 
\ohead{  \ifstr{\rightbotmark}{\leftmark}{}{\rightbotmark}}
\ihead{\leftmark}
\rohead*{\pagemark}
\rehead*{\pagemark}
\lehead{\headmark}
\lohead{\headmark}

\usepackage[dvipsnames]{xcolor}
\usepackage{hyperref}
\hypersetup{colorlinks,citecolor={NavyBlue}}

\clearscrheadfoot
\cofoot[\pagemark]{\pagemark}

\makeatletter

\providecommand*{\rightbotmark}{\expandafter\@rightmark\botmark\@empty\@empty}
\makeatother
\makeindex

\usepackage{mathabx}

%% file: intro.tex
Deformation Quantization is the mathematical subject that seeks to deform a commutative algebra $A$, typically the algebra of functions on a symplectic or Poisson manifold, into a non-commutative algebra $A_\hbar$ with non-commutativity to first order determined by the Poisson structure.\footnote{This is motivated from physics. A symplectic manifold describes the phase space of a mechanical system, like $T^*\mathbb{R}^d$ for a particle in $\mathbb{R}^d$ according to Newton's law. Its algebra of smooth functions, the classical observables, would in above notation be $A$. Physicists want to associate a new object, a Hilbert space with its associated endomorphisms algebra, whose non-commutativity should be measured by the Poisson structure; think of the Heisenberg relation $[p,q]=\hbar$ with $p$ and $q$ conjugate for the symplectic structure. This endomorphisms algebra (in above notation denoted~$A_\hbar$) would be called a (deformation) quantization of algebra $A$.} It was shown by Fedosov that a deformation quantization for $A=C^\infty(M)$ always exists when $M$ is a symplectic manifold \cite{fe94}. Kontsevich generalized this in his seminal work \cite{kon03} to Poisson manifolds by using the following theorem, as we will explain momentarily.
\begin{thm_nn}[\cite{kon03}]
Given a manifold $M$ there is a quasi-isomorphism of homotopy Lie algebras
\begin{equation}\label{KF}
\mathcal{U}:\ \Big(PV^*(M)[1],0,[\_,\_]_S\Big)\rightarrow \Big(CH^*_{loc}\big(C^\infty(M)\big)[1],d_H,[\_,\_]_G\Big)
\end{equation}
between the polyvector fields Lie algebra with Schouten  bracket and the polydifferential Hochschild cochains Lie algebra with Hochschild differential and Gerstenhaber bracket.
    \end{thm_nn}
A homotopy Lie algebra has an associated `space' cut out by its Maurer-Cartan equation and a quasi-isomorphism of homotopy Lie algebras induces an equivalence between the associated Maurer-Cartan spaces. The associated space to the polyvector field Lie algebra encodes  possible Poisson structures on $M$ and the space associated to the Hochschild cochains Lie algebra all possible non-commutative deformations of the algebra of functions of $M$. This allowed Kontsevich to conclude from his theorem above mentioned description of deformation quantizations.

\medskip
The map of Kontsevich's theorem is extremely complicated and appears miraculous, written as an infinite sum over graphs and induced contractions of tensors. A successful and influential research effort allowed to understand in hindsight Kontsevich's work via the notion of operads \cite{kon99op,ta03formality, will15} etc. However, Kontsevich's initial inspiration came from string theory, which was later elaborated on by Cattaneo-Felder \cite{cattaneo2000path} using physical methods.
{One of our key motivations is to deepen this line of thought, but from a different angle. We would like to explain in a precise mathematical sense Kontsevich's deformation quantization from the perspective of string (field) theory and Sen-Zwiebach's open-closed string vertices \cite{SZ94,Zw98} and to connect with the operadic ideas, already alluded to by Kajiura-Stasheff in \cite{kajiura2006open} section~6.} 
More precisely, since Kontsevich's work over 30 years ago so called \emph{topological} string theories have been put on solid mathematical ground - a topological string theory can be described by a Calabi-Yau (CY) category \cite{Kon94, Cos07a,Lu09}\footnote{Or perhaps we should better say the worldsheet theory of a topological string theory, in the sense that the objects of the CY category describe the boundary conditions of open strings, the homomorphisms the open string Hilbert space and the Hochschild homology the closed string Hilbert space.} - and we place ourselves in this setting; this is different than the original setting, but related (proposition \ref{commdia}) and hopefully instructive. The prototypical example of a CY category is the category of coherent sheaves of a CY manifold (or a smooth proper CY variety over $\mathbb{C}$ \cite{Serre56}). In this note, to kick-off the efforts to our motivational question, we record the CY manifold version of Kontsevich's morphism \eqref{KF}, which to our knowledge is missing from the literature.
\begin{thm_nn_a}
    Given a $d$-dimensional compact Calabi-Yau manifold $X$ there is a quasi-isomorphism of homotopy Lie algebras (up to shifts)
\begin{equation}\label{JU}
\mathcal{U}_{CY}:\ \Big(PV^{*,*}(X)\llbracket t\rrbracket,\bar{\partial}+t\partial,[\_,\_]_S\Big)\rightarrow \Big(Cyc^*_{loc}\big(\Omega^{0,*}(X)\big),d_{cyc},\{\_,\_\}_o\Big)
\end{equation}
between the Calabi-Yau versions of the polyvector field Lie algebra and the Hochschild cochain Lie algebra, whose underlying chain complexes compute the cyclic cohomology of $Coh(X)$ \cite{kell07}. 
\end{thm_nn_a}
 \emph{BCOV theory.} The Lie algebra on the left of map \eqref{JU} appears in Barannikov-Kontsevich's work on variation of Hodge structures and the extended moduli space of complex structures \cite{barannikov1997frobenius,barannikov1999generalized}. In the context of B-model enumerative invariants in Costello-Li's work \cite{coli12,CL15} this Lie algebra describes for $d=3$ the classical closed string field theory of the B-model, also known as Kodaira–Spencer gravity \cite{BCOV94} or in general as BCOV theory.

 \smallskip
  \emph{Holomorphic Chern-Simons.} The Lie algebra on the right of map \eqref{JU} (a cousin of the Hochschild Lie algebra, see below) is relevant to open B-model enumerative invariants \cite{Ul25}. Interestingly, this Lie algebra also describes classical observables of large $N$ holomorphic Chern-Simons theory \cite{CL15,GGHZ21, ha25b}, for $d=3$ also known as the open string field theory of the B-model. For instance, if $V\rightarrow X$ is a vector bundle (so determining an object in $Coh(X)$), we can combine the trace map and the LQT map at rank $N$ to obtain a map of homotopy Lie algebras
 \begin{equation}\label{vb}
  \Big(Cyc^*_{loc}\big(\Omega^{0,*}(X)\big),d_{cyc},\{\_,\_\}_o\Big)\rightarrow \Big( \mathcal{O}\big(\Omega^{0,*}(X,End(V^{\oplus N}))[1]\big),\{S_{hCS},\_\},\{\_,\_\}\Big).  
 \end{equation}
Here the right hand side describes classical observables of holomorphic Chern-Simons theory. Those are functions on the space of Dolbeault forms with values in the endomorphisms of the vector bundle $V^{\oplus N}$, together with the classical BV bracket and the differential given by bracketing with the holomorphic Chern-Simons functional, see eg. \cite{CL15}. Note that a vector bundle $V\rightarrow X$ corresponds to an object $\mathcal{F}\in Coh(X)$ called a space-filling brane. In the case of a general brane $\mathcal{F}\in Coh(X)$, eg. one supported on a curve, the open string field theory can be computed from the Calabi-Yau algebra $\mathbb{R}Hom(\mathcal{F},\mathcal{F})$, which can in general be quite complicated, see \cite{cg21} appendix E.1 for an example. However, just by abstract reasons we also get a map like \eqref{vb} to the classical observables of the open string field theory of such $\mathcal{F}$. In the case that an object $\mathcal{F}\in Coh(X)$ is a generator then for $N\rightarrow \infty$ the corresponding map like \eqref{vb} becomes a quasi-isomorphism \cite{LQ84,Tsy83}.

 \smallskip
\emph{Open-closed duality.} Importantly, the map \eqref{JU} thus exhibits a closed string - open string duality.\footnote{This terminology should be used only for $d=3.$ Davide Gaiotto pointed out to me that one should not confuse the terms (topological) string theory and 2d dg-TFT of Calabi-Yau type. In particular theorem A is about the 2d dg-TFT B-model, which for $d=3$ implies some result about the corresponding classical
string theory.} One may expect that the mirror A-model phenomena to theorem A are the bulk-boundary deformations discussed in~\cite{fooo11}. Lastly, the linear part of $\mathcal{U}_{CY}$ has been used crucially in \cite{CL15} to couple BCOV theory with holomorphic Chern-Simons and to argue the existence of a unique quantization, leading to a definition of higher genus B-model invariants. 

 \smallskip
 We elaborate on advocated perspective and the motivational question in follow-up work~\cite{Ul26b}.

\medskip
We can generalize theorem A to non-compact CY manifolds at the cost of working with a different Lie algebra $(CH^*_{loc}(\Omega^{0,*}(X))^{\sigma},d_H,[\_,\_]_G)$ on the right hand side (isomorphic to the original one in the compact case), a certain sub Lie algebra of the Hochschild cochain Lie algebra.

 A Maurer-Cartan element of $(\hbar CH^*_{loc}(\Omega^{0,*}(X))^{\sigma}\llbracket\hbar\rrbracket,d_H,[\_,\_]_G)$ deforms the dg-algebra  $\Omega^{0,*}(X)\llbracket\hbar\rrbracket$ to an $A_\infty$-algebra with for $n\in\mathbb{N}$ higher multiplications 
    \begin{equation}\label{hmsa}
    \Tilde{m}_n:\Omega^{0,*}(X)\llbracket\hbar\rrbracket^{\otimes n}\rightarrow \Omega^{0,*}(X)\llbracket\hbar\rrbracket[1-n] \end{equation}
    that satisfy 
        \begin{equation}\label{wsdsinon}
        \int_X\Tilde{m}_n(f_1,\cdots,f_n)g\omega=\pm \int_X\Tilde{m}_n(f_2,\cdots,f_n,g)f_1\omega\end{equation}
        for all compactly supported Dolbeault forms $f_i,g\in \Omega^{0,*}_c(X).$ We call those deformations of order $\hbar>0$, following \cite{wica12} section~8, closed (with respect to the volume form) deformation quantizations of the Dolbeault resolved algebra of holomorphic functions. If $X$ is compact it follows as in \cite{kon03,wica12} from theorem A:
    \begin{corollary_nn_b}
 Given a compact CY manifold $X$ there is a family of proper CY $A_\infty$-algebras on  $\Omega^{0,*}(X)\llbracket\hbar\rrbracket$ deforming to order $\hbar>0$ the initial dg algebra structure (but with standard integration pairing), parametrized by the Maurer-Cartan space of $\big(\hbar PV^{*,*}(X)\llbracket t\rrbracket\llbracket\hbar\rrbracket,\bar{\partial}+t\partial,[\_,\_]_S\big)$.
    \end{corollary_nn_b}

As a special case and an application to deformation quantization we find: 
    \begin{corollary_nn_c}
For every Calabi-Yau unimodular holomorphic Poisson manifold there is a canonical closed deformation quantizations of its Dolbeault resolved algebra of holomorphic functions, which coincides with (the complex version of) Kontsevich's deformation quantization.
\end{corollary_nn_c}
    Indeed, the quasi-isomorphism from \eqref{JU} restricts to a morphism injective on homology (up to some shifts)
\begin{equation}\label{FU}
\mathcal{U}_{CY}:\ \Big(ker_\partial PV^{*}_{hol}(X),0,[\_,\_]_S\Big)\rightarrow \Big( CH^*_{loc}\big(\Omega^{0,*}(X)\big)^\sigma,d_H,[\_,\_]_G\Big);
\end{equation}
 see proposition \ref{commdia} for the relation to Kontsevich's original morphism. 
 Maurer-Cartan elements in domain of \eqref{FU} are divergence free (also called unimodular) holomorphic Poisson structure.

\medskip
The proof of theorem A is just a combination of known results from \cite{wica12} on a cyclic formality morphism for manifolds with a volume form and the results of \cite{cal05} on a formality morphism for complex manifolds via the Dolbeault resolution; no new arguments are necessary. One first proves the theorem for $\mathbb{C}^d$ with the standard Calabi-Yau form, which follows from base change of the local version of \cite{wica12}. One globalizes using Fedosov-Dolgushev resolutions \cite{do05}, taking care regarding the holomorphic structure as explained in \cite{cal05}. For the globalization we have to pick a torsion free connection, which in Kontsevich's original work presented an additional degree of freedom. In the CY setting the connection has to be compatible with the complex structure and the CY form, which uniquely determines it; it is the Chern connection (equal to the  Levi-Civita connection) of the CY manifold. 

\medskip
\textbf{Acknowledgments.} I would like to thank Junwu Tu for a discussion that let me to write down this note, whose results were also anticipated in Costello-Li's work \cite{CL15}, and Davide Gaiotto for feedback. The proof ideas are due to Kontsevich, Dolgushev, Calaque-Dolgushev-Halbout and Willwacher-Calaque, we just put them together.

%% file: cyform.tex
\section{Calabi-Yau dg Lie Algebras}\label{cydgla}
A $d$-dimensional Calabi-Yau manifold $(X,\omega)$ is a $d$-dimensional complex manifold $X$ with a nowhere vanishing $(d,0)$-form $\omega\in \Omega^{d,0}(X)$. In this section we introduce dg Lie algebras naturally associated to a Calabi-Yau manifold.
\subsection{Polyvector fields}\label{pvfields}
We denote the Dolbeault resolution of  polyvector fields by
    $$PV^{*,*}(X)=\Omega^{0,*}(X,\wedge^* TX).$$
Endowed with the differential $\bar\partial:PV^{*,*}(X)\rightarrow PV^{*,*+1}(X)$ and the wedge product it becomes a differential graded algebra. Contracting with $\omega\in \Omega^{d,0}(X)$ leads to an isomorphism
$$PV^{*,*}(X)\cong \Omega^{d-*,*}(X).$$
We denote by abuse of notation the differential on $PV^{*,*}(X)$ induced under this isomorphism from $$\partial: \Omega^{*,*}(X)\rightarrow \Omega^{*+1,*}(X)$$ by the same letter. It turns out that $\partial$ is not a derivation with respect to the algebra structure on $PV^{*,*}(X)$, but defines a shifted BV algebra, whose associated bracket we denote by $[\_,\_]_S$. One can check that this is just the standard Schouten bracket on polyvector fields. See eg. \cite{coli12} for more details. 
\begin{df}\label{cypv}
    The Calabi-Yau polyvector field differential graded Lie algebra is
    $$\Big(PV^{*,*}(X)\llbracket t\rrbracket[1],\bar\partial+t\partial,[\_,\_]_S\Big),$$
    where $t$ is of cohomological degree 2 and we extended the Schouten bracket $t$-linearly.
\end{df}
\subsection{Hochschild and cyclic cochains}\label{cychochdgla}
In definition \eqref{hCS} we set up the 'cyclic cohomology' analogue of the Gerstenhaber dg Lie algebra on Hochschild cochains, for a \emph{compact} CY manifold. In definition \eqref{dchcc} we find a replacement for this dg Lie algebra in the non-compact case, as we explain in lemma \eqref{iitcc}.

\medskip
Given a topological vector space $V$ obtained as the space of sections of a $\mathbb{Z}$-graded vector bundle $E\rightarrow X$ we denote by 
$$Cyc_{loc}^*(V)$$
the space of cyclically invariant local functionals on $V[1]$, see definition 3.10 of \cite{ha25b} and section 2.4 of \cite{CL15}. If $$V=\Omega^{0,*}(X)$$ for $(X,\omega)$ a compact $d-$dimensional CY manifold then $Cyc_{loc}^*(V)$
can be endowed with a $(d-2)$-shifted Lie bracket $\{\_,\_\}_o$ obtained from the integration pairing
\begin{equation}\label{pairing}
\Omega^{0,*}(X)\otimes \Omega^{0,*}(X)\rightarrow \mathbb{C},\ \ \alpha\otimes\beta\mapsto \int_X(\alpha\wedge\beta\wedge\omega),
\end{equation}
see section 2.4 of \cite{CL15} or more generally definition 3.7 and proposition 3.11 of \cite{ha25b}.\footnote{See also section 3.2 of \cite{GGHZ21}, \cite{Ul25a} and references therein making clear the ‘open string' nature of this bracket (for the finite-dimensional, cyclic $A_\infty$-algebra setting).} The dg algebra structure $\big(\Omega^{0,*}(X),\bar{\partial},\wedge\big)$ induces a Maurer-Cartan element, the holomorphic Chern-Simons functional, of the $(d-2)$-shifted Lie algebra $\big(Cyc^*(\Omega^{0,*})_{loc},\{\_,\_\}_o\big)$  
\begin{equation}\label{hCSfun}
S_{hCS}\in MCE\Big(Cyc_{loc}^*(\Omega^{0,*}(X)),\{\_,\_\}_o\Big),
\end{equation}

see again section 2.4 of \cite{CL15} and section 3.2 of \cite{ha25b} for a general treatment. We can twist the Lie algebra $\big(Cyc^*(\Omega^{0,*})_{loc},\{\_,\_\}_o\big)$ by the Maurer-Cartan element \eqref{hCSfun} and denote the resulting differential by 
$$d_{cyc}:=\{S_{hCH},\_\}_o.$$
Finally shifting by $(d-2)$, gives us the dg Lie algebra that we are interested in.
\begin{df}\label{hCS}
    Given a compact $d$-dimensional CY manifold $(X,\omega)$ we denote by 
    $$\Big(Cyc^*_{loc}\big(\Omega^{0,*}(X)\big)[2-d],d_{cyc},\{\_,\_\}_o\Big)$$
    the dg Lie algebra of cyclically invariant local functionals on Dolbeault forms together with the Hochschild type differential and the open string bracket.
\end{df}

\smallskip
We would like to define a variant of the dg Lie algebra from definition \ref{hCS} for non-compact CY manifolds, but run into the problem that the pairing \eqref{pairing} used to define the Lie bracket doesn't make sense in the non-compact case. However, the existence of a volume form allows to define a different `cyclic Hochschild cochain' complex (definition \ref{dchcc}), which turns out to be isomorphic to the one from definition \ref{hCS} in the compact case (lemma \ref{iitcc}).

\medskip
Recall that given a dg-algebra (or $A_\infty$-algebra) $A$ its Hochschild cochain complex is defined as 
$$CH^*(A)=\prod_{k=1}^\infty Hom(A[1]^{\otimes k},A),$$
together with the Hochschild differential $d_H$ see eg. section 2 of \cite{wica12}. Further $\big(CH^*(A)[1],d_H\big)$ can be endowed with a dg Lie algebra structure via the Gerstenhaber bracket $[\_,\_]_G$, see eg. section 2 of \cite{wica12}. Given $M$ a smooth manifold we can define the subcomplex of Hochschild cochains of $A=C^\infty(M)$ generated by $A$ and differential operators under the cup product; the Hochschild differential and Gerstenhaber bracket descend to these spaces (see eg. section 2.1 of \cite{wica12}). Given a complex manifold $X$ we can make the same definition both for the algebra of holomorphic functions $A=\mathcal{O}_{hol}(X)$ respectively for the dg-algebra of Dolbeault forms $A=\Omega^{0,*}(X)$, which we summarize in the following definition.
\begin{df}
    Given a complex manifold denote the dg Lie algebras of polydifferential (or local) operators of the algebra of holomorphic functions  by $$\Big(CH^*_{loc}\big(\mathcal{O}_{hol}(X)\big)[1],d_H,[\_,\_]_G\Big)$$
    respectively the dg Lie algebras of polydifferential (or local) operators of the dg-algebra of Dolbeault forms by 
 $$\Big(CH^*_{loc}\big(\Omega^{0,*}(X)\big)[1],d_H,[\_,\_]_G\Big).$$
     \end{df}
    \smallskip
    
To further tackle the non-compact case we make an auxiliary construction.
\begin{cons}\label{nfcon}
We define auxiliary graded vector spaces, exactly as in section 2.2 of \cite{wica12}, but use a different notation.
   Given a (not necessarily compact) $d$-dim. CY manifold $(X,\omega)$ one sets
\begin{equation}\label{ibidbnf}
pCyc^n(\mathcal{O}_{hol}(X)):=\big(CH^{n+1}_{loc}(\mathcal{O}_{hol}(X))[1]\otimes_{\mathcal{O}_{hol}} \Omega^{d,d}(X)\big)_{T^1(X)}\end{equation}
and
\begin{equation}\label{eiwvzw}
pCyc^*(\Omega^{0,*}(X)):=\big(CH^{n+1}_{loc}(\Omega^{0,*}(X))[1]\otimes_{C^\infty(X)} \Omega^{d,0}(X)\big)_{T^1(X)}.
\end{equation}
These chain complexes have an action by the group $\mathbb{Z}_{n+1}$, whose generator we denotes by $\sigma$. As in section 2.2 of \cite{wica12} we have isomorphisms of graded vector spaces
\begin{equation}\label{blabla}
\phi:CH^{*}_{loc}\big(\mathcal{O}_{hol}(X)\big)[1]\cong pCyc^*\big(\mathcal{O}_{hol}(X)\big)[2d-1] \ \ \text{and} \ \ CH^{*}_{loc}\big(\Omega^{0,*}(X)\big)[1]\cong pCyc^*\big(\Omega^{0,*}(X)\big)[d-1],\end{equation}
induced from cup product with the identity operator and tensoring with $\omega\wedge \bar{\omega}$ respectively $\omega$. 
\end{cons}

The isomorphisms \eqref{blabla} allow us to define the cyclic group action generated by $\sigma$ also on $CH^{*}_{loc}(\mathcal{O}_{hol}(X))$ and $CH^{*}_{loc}(\Omega^{0,*}(X))$. We call the invariants under this action $\omega$-cyclic. As in proposition 6 and 7 of \cite{wica12} it follows that the space of $\omega$-cyclic Hochschild cochains is closed under Hochschild differential and Gerstenhaber bracket, leading to following definition.
\begin{df}\label{dchcc}
Given a (not necessarily compact) CY manifold $(X,\omega)$ denote by
     $$\Big(CH^{*}_{loc}\big(\mathcal{O}_{hol}(X)\big)[1]^\sigma,d_H,[\_,\_]_G\Big),$$
     the $\omega$-cyclically invariant polydifferential operator dg Lie algebra of the algebra of holomorphic functions respectively by 
     $$\Big(CH^{*}_{loc}\big(\Omega^{0,*}(X)\big)[1]^\sigma,d_H,[\_,\_]_G\Big)$$
      the $\omega$-cyclically invariant polydifferential operator dg Lie algebra of the algebra of the dg algebra of Dolbeault forms.
\end{df}
The two dg Lie algebras from definition \eqref{dchcc} are related through following lemma.
\begin{lemma}\label{ffhpoin}
    Given a CY manifold $(X,\omega)$ there is a quasi-isomorphism of dg Lie algebras
    $$\Omega^{0,*}(X,\Big(CH^*_{loc}\big(\mathcal{O}_{hol}(X)\big)[1]^\sigma,d_H,[\_,\_]_G\Big))\rightarrow \Big(CH^*_{loc}\big(\Omega^{0,*}(X)\big)[1]^\sigma,d_H,[\_,\_]_G\Big).$$
\end{lemma}
\begin{proof}
 We can define in a straightforward way fine sheaves of dgla's  $\mathcal{O}_1$ and $\mathcal{O}_2$ whose value on an open subset $V\subset X$ is  $$\mathcal{O}_1(V)=\Omega^{0,*}(V,\Big(CH^*_{loc}\big(\mathcal{O}_{hol}(V)\big)[1]^\sigma,d_H,[\_,\_]_G\Big))$$
 respectively
 $$\mathcal{O}_2(V)=\Big(CH^*_{loc}\big(\Omega^{0,*}(V)\big)[1]^\sigma,d_H,[\_,\_]_G\Big).$$
 The non-derived global sections of $\mathcal{O}_1$ and $\mathcal{O}_2$, equal to its derived global sections since the sheaves are fine, give the dgla's in the lemma. There is a zig-zag of sheaves of dg Lie algebra
 $$\mathcal{O}_1\leftarrow \Big(CH^*_{loc}(\mathcal{O}_{hol})[1],d_H,[\_,\_]_G\Big)\rightarrow \mathcal{O}_2.$$

 Due to the $\bar{\partial}$-Poincaré lemma those maps of sheaves of dg Lie algebras  are locally quasi-isomorphisms; or in other words $\mathcal{O}_1$ and $\mathcal{O}_2$ are equivalent in the homotopy category. The lemma then follows since taking derived global sections sends equivalence to equivalence.   
\end{proof}

Finally we can explain how the two 'cyclic' dg Lie algebras on Hochschild cochains are related:
\begin{lemma}\label{iitcc}
 Given a compact $d$-dim. CY manifold $(X,\omega)$ there is an iso of dg Lie algebras $$F: \Big( CH^{*}_{loc}\big(\Omega^{0,*}(X)\big)[1]^{\sigma},d_H,[\_,\_]_G\Big)\cong\Big(Cyc^*_{loc}\big(\Omega^{0,*}(X)\big)[d-2],d_{cyc},\{\_,\_\}_o\Big)$$
 induced by integration
$$f\mapsto \Big( F(f):(a_1,\cdots ,a_n)\mapsto \int_X\phi(f)(a_1,\cdots, a_n)\Big),$$
where we used isomorphism \eqref{blabla} and implicitly the wedge product.
\end{lemma}
\begin{proof}
    This follows from unraveling the definitions.
\end{proof}
\section{Dolgushev-Fedosov resolutions}\label{dfs}
We explain, following \cite{kon03, do05,cal05,wica12}, how we can embed the non-linear (over $C^\infty(X)$) CY dg Lie algebras from definition \ref{cypv} and \ref{dchcc} into linear (over $C^\infty(X)$) dg Lie algebras, which is done via Gelfand-Kazhdan formal geometry or more precisely using Dolgushev-Fedosov resolutions. This is the important tool to globalize the local CY deformation quantization morphism from the next section.

\medskip
Given a complex manifold we recall from \cite{cal05}, section 5 (applied to the canonical holomorphic Lie algebroid $T^{1,0}(X)$) 
\begin{itemize}
\item the sheaf of algebras $\mathcal{O}^f_{hol}$ of formal functions on $T^{1,0}(X)$ tangent to the fibers, 
\item the sheaf of dgla's $\big(\mathcal{T}^*_{hol}[1],0,[\_,\_]_S^f\big)$ of formal polyvector fields on $T^{1,0}(X)$ tangent to the fibers, 
\item the sheaf of dgla's $\big(\mathcal{D}^*_{hol}[1],d_H^f,[\_,\_]_G^f\big)$ of formal polydifferential operators on $T^{1,0}(X)$ tangent to the fibers.
\item the sheaf of dg algebras $\big(\mathcal{A}^*_{hol},d_{dR}^f,\wedge\big)$ of formal de Rham forms on $T^{1,0}(X)$ tangent to the fibers.
\end{itemize}
We consider the respective dgla's (and dg algebra) obtained by taking de Rham forms with values in mentioned dgla's (dg algebra). The Fedosov differential (theorem 5.9 of \cite{cal05}, following \cite{do05}) allows to perturb the resulting dgla's (and dg algebra), giving us the dgla's (and the dg algebra) on the right side of following theorem. 
\begin{theorem}[\cite{cal05}]\label{holdgr}
Given an affine torsion free connection $\nabla$ compatible with the complex structure there are quasi-isomorphisms of dg Lie algebras 
\begin{equation}\label{dfrpf}
\lambda_{\nabla}:\ \Big(PV^{*,*}(X)[1],\bar{\partial},[\_,\_]_S\Big)\rightarrow \Big(\Omega^*(X,\mathcal{T}^*_{hol}[1]),D,[\_,\_]_S^f\Big)\end{equation}
\begin{equation}\label{dfrpo}
\lambda_{\nabla}:\ \Omega^{0,*}(X,\Big( CH^{*}_{loc}(\mathcal{O}_{hol}(X))[1],d_H,[\_,\_]_G\Big))\rightarrow \Big(\Omega^*(X,\mathcal{D}^*_{hol}[1]),d_H^f+D,[\_,\_]_G^f\Big)\end{equation}
and a quasi-isomorphism of dg algebras
\begin{equation}
\lambda_{\nabla}:\ \Big(\Omega^{*,*}(X),\partial+\bar{\partial},\wedge\Big)\rightarrow \Big(\Omega^*(X,\mathcal{A}^*_{hol}),\partial^f+D,\wedge\Big).\end{equation}
\end{theorem}
\begin{proof}
Indeed, this follows from proposition 5.11 and 5.12 of \cite{cal05} (respectively the sentences just below) applied to the canonical holomorphic Lie algebroid $T^{1,0}(X).$  
\end{proof}

The goal of this section is to explain how to adapt the Dolgushev-Fedosov dg Lie algebras to the Calabi-Yau setting by mimicking the definitions from section \ref{cydgla} and how the Dolgushev-Fedosov resolutions - the content of theorem \ref{holdgr} - work in the CY setting. The arguments in this section are the same as those in section 7 of \cite{wica12}, we just elaborate a bit more. 

\subsection{Formal polyvector fields}

\medskip
 Given a $d$-dimensional Calabi-Yau manifold $(X,\omega)$, we can interpret $\omega\in \mathcal{A}^d_{hol}$, with constant formal coefficients. Contracting with $\omega\in \mathcal{A}^*_{hol}$ induces an isomorphism  
 $$\iota_\omega:\mathcal{T}^*_{hol}\rightarrow  \mathcal{A}^*_{hol}$$
 and we denote by abuse of notation by $\partial^f$ the differential on $\mathcal{T}^*_{hol}$ corresponding under this isomorphism to the  $\partial^f$ differential on $\mathcal{A}^*_{hol}$. As in section \ref{pvfields} this give us following.
\begin{df}\label{cyfpvf}
 The Calabi-Yau version of the dgla sheaf of formal polyvector fields is  $$\Big(\mathcal{T}^*_{hol}\llbracket t\rrbracket[1]),t\partial^f,[\_,\_]_S^f \Big),$$
 where $t$ is of cohomological degree 2.
\end{df}
Recall that for theorem \ref{holdgr} we considered de Rham forms with values in the formal dgla and perturbed the resulting dgla by the Fedosov differential. Here, we additionally need to argue that the Fedosov differential commutes with the $t\partial^f$ differential. However, assuming that $\nabla(\omega)=0$, this follows exactly as in proposition 29 of \cite{cal05}, giving us the dgla on the right of following theorem. 
\begin{theorem}\label{polyvctreso}
    Given a Calabi-Yau manifold $X$ there is a quasi-isomorphisms of dg Lie algebras (associated to the Chern connection $\nabla$) 
\begin{equation}\lambda_\nabla: \Big(PV^{*,*}(X)\llbracket t\rrbracket[1],\bar{\partial}+t\partial,[\_,\_]_S\Big)\rightarrow \Big(\Omega^*(X,\mathcal{T}^*_{hol}\llbracket t\rrbracket[1]),D+t\partial^f,[\_,\_]_S^f\Big).\end{equation}
\end{theorem}
\begin{proof}
We define the map $\lambda_\nabla$ by $t$-linearly extending the map \eqref{dfrpf}. This gives us a well defined map of dgla if we can argue that following diagram \eqref{slmggga} commutes. Then the fact that we obtain a quasi-isomorphism follows from a standard spectral sequence argument and theorem~\ref{holdgr}.

Indeed, since $\omega$ is flat with respect to a connection $\nabla$ we claim to obtain a commutative diagram of chain complexes
 \begin{equation}\label{slmggga}
\begin{tikzcd}
 \Big(PV^{*,*}(X),\partial\Big)\arrow{r}{ \lambda_{\nabla}}\arrow{d}{\iota_\omega}& \Big(\Omega^*(X,\mathcal{T}^*_{hol}),\partial^f\Big)\arrow{d}{\iota_\omega}\\
  \Big(\Omega^{*,*}(X),\partial\Big)\arrow{r}{ \lambda_{\nabla}}& \Big(\Omega^*(X,\mathcal{T}^*_{hol}),\partial^f\Big).
 \end{tikzcd}
\end{equation}
Notice that if we can argue that the diagram commutes in graded vector space the claim follows, since the vertical maps are isomorphisms and we know that all but the upper horizontal map commutes with the differentials. Notice further that if we can argue that $\lambda_{\nabla}(\omega)=\omega$ commutativity of the diagram follows, since the maps $\lambda_{\nabla}$ respect the structure of dgla modules of forms over the dgla's of polyvector fields (proposition 5.11 of \cite{cal05}), ie. the maps $\lambda_{\nabla}$ are compatible with the contraction morphisms. We recall that in local coordinates $(z_i\bar{z}_i)$ the map $\lambda_{\nabla}$ is to leading order in the formal variables $(y_i)$ defined as
\begin{equation}\label{above}
    \lambda_{\nabla}(\omega)=\omega+y^i\frac{\partial}{\partial z_i} \omega+y^i\Gamma_{ik}^k\cdot \omega +\mathcal{O}(y^2),
\end{equation}    
where the third summand involving the Christoffel symbols $\Gamma_{ij}^k$ of the connection $\nabla$ comes from the action by Lie derivative of vector fields on forms compared to formula 5.23 of \cite{cal05}.
In general $\lambda_{\nabla}$ is defined iteratively in powers of $(y_i)$, which means in particular that $\lambda_{\nabla}(\omega)=\omega$ if in equation \eqref{above} the terms linear in $(y_i)$ are zero. However, this is exactly the condition that $\nabla\omega=0$ written in coordinates, so that we can conclude commutativity of diagram \eqref{slmggga}.     
\end{proof}

\subsection{Formal polydifferential operators}
 Given a $d$-dimensional Calabi-Yau manifold $(X,\omega)$, we can interpret $\omega\wedge\bar{\omega}\in \mathcal{A}^{d,d}$.  Let us define the formal version of the cyclic Hochschild cochain dg Lie algebra from section \ref{cychochdgla}. As in construction \ref{nfcon} one can define 
$$ pCyc^n(\mathcal{O}_{hol}^f):=\big(\mathcal{D}^{n+1}_{hol}[1]\otimes_{\mathcal{O}_{hol}^f} \mathcal{A}^{d,d}\big)_{\mathcal{T}^1(X)}$$
and show that there is an isomorphism
$$\phi^f:\ \mathcal{D}^{n}_{hol}[1] \cong pCyc^n(\mathcal{O}_{hol}^f)[2d-1],$$
induced by cup product with the identity operator and tensoring with $\omega\wedge\bar{\omega}$.
This isomorphism allows to define an action by the cyclic group on formal Hochschild cochains whose invariants we denote by $\mathcal{D}^{*}_{hol}[1]^\sigma$.  As in section \ref{cychochdgla} one shows that the fiberwise Hochschild differential and Gerstenhaber bracket leave invariant the cyclic action, which gives us following definition.
\begin{df}\label{123467}
 The Calabi-Yau version of the dgla sheaf of formal polydifferential operators is  
$$
    \Big(\mathcal{D}^{*}_{hol}[1]^\sigma,d_H^f,[\_,\_]_G^f\Big).
$$
 \end{df}
 
 \medskip

Again, we would like to consider de Rham forms with values in the dgla from definition \ref{123467} and then to perturb by the Fedosov differential. For this we need to argue that the Fedosov differential leaves invariant the space of de Rham forms with values in formal cyclic Hochschild cochains. For that we recall from \cite{cal05} that we can write the Fedosov differential
$$D_\nabla:\ \Omega^*(\mathcal{D}^{*}_{hol}[1])\rightarrow \Omega^{*+1}(\mathcal{D}^{*}_{hol}[1])$$
 as
$$D_\nabla=d_{dR}+[Q,\_]_G$$
where $Q$ is a vector valued one form. As in proposition 29 of \cite{wica12} one has that $div(Q)=0$ since $\nabla\omega=0$. Now let us assume that $\alpha\in \Omega^*(\mathcal{D}^{n}_{hol}[1]^\sigma)$, which means that ${[id \cup \alpha \otimes \omega\wedge\bar{\omega}]}$ is cyclically symmetric as a (de Rham valued) formal $(n+1)$-Hochschild cochain tensored with the volume form modulo the action of formal vector fields. Then we want to show that  ${D_\nabla\alpha\in \Omega^{*+1}(\mathcal{D}^{n}_{hol}[1]^\sigma)},$
or in other words that 
$$[id \cup ((d_{dR}+[Q,\_]_G )\alpha) \otimes \omega\wedge\bar{\omega} ]$$
is cyclically symmetric in the above sense. We  compute
\begin{align*}
&[id \cup ((d_{dR}+[Q,\_]_G )\alpha) \otimes \omega\wedge\bar{\omega} ] \\
=&[d_{dR}(id \cup  \alpha) \otimes \omega\wedge\bar{\omega} ]+[id \cup  \alpha\{Q\} \otimes \omega\wedge\bar{\omega} ]+[id \cup  \{Q\}\alpha \otimes \omega\wedge\bar{\omega} ]\\
=&[d_{dR}(id \cup  \alpha) \otimes \omega\wedge\bar{\omega} ]+[id \cup  \alpha\{Q\} \otimes \omega\wedge\bar{\omega} ]-[id\{Q\} \cup  \alpha \otimes \omega\wedge\bar{\omega} ]\\
=&[(d_{dR}+(\_)\{Q\})(id \cup  \alpha) \otimes \omega\wedge\bar{\omega} ]
\end{align*}
where the second line follows since the de Rham differential acts trivial on the identity operator and by definition of the Gerstenhaber bracket in terms of the brace operation, see eg. section 2 of \cite{wica12}.  To get from the second to third line we used that
$$[id \cup  \{Q\}\alpha \otimes \omega\wedge\bar{\omega} ]=-[ id\{Q\} \cup  \alpha \otimes \omega\wedge\bar{\omega} ]-[ id \cup  \alpha \otimes Q(\omega\wedge\bar{\omega}) ],$$
since we are taking coinvariants with respect to the vector field action, and that  $$Q(\omega\wedge\bar{\omega}) =div(Q)(\omega\wedge\bar{\omega})=0$$
since $div(Q)=0$ as mentioned above. The element $[(d_{dR}+(\_)\{Q\})(id \cup  \alpha) \otimes \omega\wedge\bar{\omega} ]$ is inherently cyclically symmetric since obtained by acting with a vector field on the cyclically symmetric element $\alpha$. Thus the Fedosov differential leaves invariant the space of cyclic formal Hochschild cochains and the dg Lie algebra on the right of following theorem is well defined.
\begin{theorem}\label{hcochainsdfr}
Given a Calabi-Yau manifold $X$ there is a quasi-isomorphism of dg Lie algebras
\begin{equation}
\lambda_\nabla:\ \Omega^{0,*}(X,\Big(CH^{*}(\mathcal{O}_{hol})[1]^\sigma,d_H,[\_,\_]_G\Big))\rightarrow \Big(\Omega^*(X,\mathcal{D}_{hol}[1]^\sigma),d_H^f+D,[\_,\_]_S^f\Big).\end{equation}
\end{theorem}
\begin{proof}
Indeed, as in \cite{wica12} we can define the map $\lambda_\nabla$ as the restriction of the map \eqref{dfrpo} to the subspace of cyclic Hochschild cochains, which lands in formal cyclic Hochschild cochains since $\phi^f\circ \lambda_\nabla = (\lambda_\nabla\otimes \lambda_\nabla)\circ \phi$, which follows from $\lambda(\omega)=\omega$ and $\lambda(id)=id$, and since $\lambda$ commutes with the cyclic shift operator as $\lambda$ is built from derivations.

Then we can directly conclude that we obtain a quasi-isomorphism since we are in characteristic zero where taking invariants with respect to a finite group action is exact.

\end{proof}

\section{Proof of theorem A}
The crucial input for the proof of theorem A is the local version of theorem A on flat space, which follows directly from \cite{wica12} by base change. We then explain, following \cite{do05, cal05,wica12}, and via the work from section \ref{dfs} how to globalize and to obtain theorem A in general.
\subsection{Local version }
Indeed, for the local version all work is already done.
\begin{theorem}[\cite{kon03, wica12}]
There is a quasi-isomorphism of dg Lie algebras
\begin{equation}
    \Big(PV^*(\mathbb{R}^d_{\text{formal}})\llbracket t\rrbracket[1],t\partial,[\_,\_]_S\Big) \rightarrow \Big(CH^*(\mathbb{R}\llbracket x_1,\cdots x_d\rrbracket)[1]^\sigma,d_H,[\_,\_]_G\Big),
\end{equation}
 satisfying the properties from proposition 27 of \cite{wica12}.
\end{theorem}
Since $\mathbb{R}\rightarrow\mathbb{C}$ is flat it follows:
\begin{corollary}\label{localres}
    There is a quasi-isomorphism of dg Lie algebras
\begin{equation}\label{asdfgh}
    \Big(PV^*(\mathbb{C}^d_{\text{formal}})\llbracket t\rrbracket[1],t\partial,[\_,\_]_S\Big) \rightarrow \Big(CH^*(\mathbb{C}\llbracket z_1,\cdots z_d\rrbracket)[1],d_H,[\_,\_]_G\Big),
\end{equation}  
satisfying the properties from proposition 27 of \cite{wica12}.
\end{corollary}

\subsection{Globalization via Dolgushev-Fedosov}
Let us now explain how we can globalize corollary \ref{localres} to arbitrary $d$-dimensional Calabi-Yau manifold $(X,\omega)$.

\smallskip
Let $V$ denote an open disk in $X$. Then we obtain a quasi-isomorphisms
\begin{equation}\label{locally}
\mathcal{U}_{CY,V}:\ \Big(\Omega^*(V,\mathcal{T}^*_{hol}\llbracket t\rrbracket[1]),d_{dR}+t\partial^f,[\_,\_]_S^f\Big)\rightarrow \Big(\Omega^*(V,\mathcal{D}_{hol}[1]^\sigma),d_{dR}+d_H^f,[\_,\_]_S^f\Big),
 \end{equation}
since the fibers over a point of domain and codomain of \eqref{locally} coincide with domain and codomain of morphism \eqref{asdfgh} and since $V$ is contractible. 

Locally on $V$ we can write the Fedosov differential (given the Chern connection $\nabla$) as $D_\nabla=d_{dR}+Q$, where it follows as in \cite{wica12} that $Q$ satisfies the Maurer-Cartan equation of the dg Lie algebra of the domain of morphism \eqref{locally}. One has
$$\sum_{n=1}^\infty\frac{1}{n!}\mathcal{U}^n_{CY,V}(Q,\cdots Q)=\mathcal{U}^1_{CY,V}(Q)=\mathcal{U}^1_V(Q)=Q,$$
which follows, since $Q$ is locally a divergence free vector field, by property 3 and by property 2 from proposition 27 of \cite{wica12}. Thus by twisting the morphism \eqref{locally} by the Maurer-Cartan element $Q$ we obtain a quasi-isomorphism of dg Lie algebras 
\begin{equation}\label{ert}
    \Big(\Omega^*(V,\mathcal{T}_{hol}\llbracket t\rrbracket,D+t\partial^f,[\_,\_]_S^f\Big)\rightarrow \Big(\Omega^*(V,\mathcal{D}_{hol}[1]^\sigma),D+d_H^f,[\_,\_]_S^f\Big).
\end{equation}
We can extend the quasi-isomorphism \eqref{ert} to the whole manifold $X$ using a partition of unity (this is well defined as under coordinate change the vector field $Q$ changes by a vector field linear in the formal variables, see equation 58 of \cite{do05}, which vanishes under morphism \eqref{locally} by property 4 of \cite{wica12}), so that we obtain:
\begin{theorem}
    Given a Calabi-Yau manifold $(X,\omega)$ there is a quasi-isomorphism of dg Lie algebras
    \begin{equation}
    \Big(\Omega^*(X,\mathcal{T}_{hol}\llbracket t\rrbracket,D+t\partial^f,[\_,\_]_S^f\Big)\rightarrow \Big(\Omega^*(X,\mathcal{D}_{hol}[1]^\sigma),D+d_H^f,[\_,\_]_S^f\Big).
\end{equation}
\end{theorem}

\medskip
Thus putting everything together we can conclude with the main result of this note.
\begin{theorem}[theorem A]\label{iknmibm}
Given a $d$-dimensional Calabi-Yau manifold $(X,\omega)$ there is a quasi-isomorphism of dg Lie algebras
\begin{equation}
\mathcal{U}_{CY}:\ \Big(PV^{*,*}(X)\llbracket t\rrbracket[1],\bar{\partial}+t\partial,[\_,\_]_S\Big)\rightarrow \Big( CH^{*}_{loc}(\Omega^{0,*}(X)[1]^\sigma,d_H,[\_,\_]_S\Big) .\end{equation}
If $X$ is compact we obtain a quasi-isomorphism of dg Lie algebras
\begin{equation}\label{zuihjk}
\mathcal{U}_{CY}:\ \Big(PV^{*,*}(X)\llbracket t\rrbracket[1],\bar{\partial}+t\partial,[\_,\_]_S\Big)\rightarrow \Big(Cyc^*_{loc}(\Omega^{0,*}(X))[d-2],d_{cyc},\{\_,\_\}_o\Big).
\end{equation}
\end{theorem}

\medskip
\begin{proof}
  This follows from combining theorem \ref{iknmibm} above, theorem \ref{hcochainsdfr} and theorem \ref{polyvctreso} about the Dolgushev-Fedosov resolutions and lemma \ref{ffhpoin} (as well as lemma \ref{iitcc} for the compact case).
\end{proof}

We finish by recording how the morphism \eqref{zuihjk} is related to (the complex version of) Kontsevich's original morphism \cite{kon03}, that is theorem 5.2 of \cite{cal05} applied to the holomorphic Lie algebroid $T^{1,0}(X).$ 
\begin{prop}\label{commdia}
 Given a Calabi-Yau manifold the morphism from theorem A is related to (the complex version of) Kontsevich's original morphism $\mathcal{U}$ through following commutative diagram   
\begin{equation}\label{relat}
\begin{tikzcd}
\Big(PV^{*,*}(X)\llbracket t\rrbracket[1],\bar{\partial}+t\partial,[\_,\_]_S\Big)\arrow{r}{ \mathcal{U}_{CY}}& \Big(CH^*_{loc}(\Omega^{0,*}(X))[1]^\sigma,d_H,[\_,\_]_G\Big)\arrow[equal]{d}
   \\
    \Big(ker_\partial PV^{*,*}(X)[1],\bar{\partial},[\_,\_]_S\big)\arrow{r}{\mathcal{U}_{CY}|_{ker}}\arrow{u}\arrow{d}&\Big(CH_{loc}^*(\Omega^{0,*}(X)[1]^\sigma,d_H,[\_,\_]_G\big)\arrow{d} \\
    \Big( PV^{*,*}(X)[1],\bar{\partial},[\_,\_]_S\Big)\arrow{r}{\mathcal{U}}&\Big(CH_{loc}^*(\Omega^{0,*}(X)[1],d_H,[\_,\_]_G\Big),
\end{tikzcd}
\end{equation}
where the vertical morphism are the canonical inclusions and $\mathcal{U}_{CY}|_{ker}$ denotes the restriction of $\mathcal{U}_{CY}$ to divergence free polyvector fields.
\end{prop}
\begin{proof}
    We need to explain that the lower square commutes. This follows since the difference between $\mathcal{U}_{CY}|_{ker}$ and $\mathcal{U}$ is induced from their difference on a local patch (we use the same Chern connection to globalize). The only difference between the original local morphism $\mathcal{U}$ and the local morphism $\mathcal{U}_{CY}|_{ker}$ is that in the latter case we sum over Kontsevich graphs with tadpoles allowed, see section 4.2 of \cite{wica12}. 
    However, if we contract divergence free polyvector fields via a Kontsevich graphs with a tadpoles we get zero, since the contraction with a tadpole gives the divergence, compare equation 4 of \cite{wica12}. Thus only Kontsevich graphs without tadpoles contribute to $\mathcal{U}_{CY}|_{ker}$, which gives us back $\mathcal{U}$. 
\end{proof}
\begin{remark}
The morphism $\mathcal{U}_{CY}|_{ker}$ has been studied from a geometrical perspective in \cite{fe00}.
\end{remark}